\documentclass[11pt]{article}

\usepackage{fullpage}
\usepackage{authblk}
\usepackage{natbib}
\usepackage{algorithm}
\usepackage{algpseudocode}
\usepackage{amsmath,amsthm,amsfonts}
\usepackage{amsmath}
\usepackage[subnum]{cases}
\usepackage{tcolorbox}
\usepackage{mathrsfs}
\usepackage{amssymb}
\usepackage{color}
\usepackage{mathrsfs}
\usepackage{enumitem}
\usepackage{bm}
\usepackage{multirow}
\usepackage{booktabs}
\usepackage{makecell}
\usepackage{graphicx}
\usepackage{epstopdf}
\usepackage{subcaption}
\usepackage{comment}
\usepackage{cases}
\usepackage{appendix}
\usepackage{tikz}
\usetikzlibrary{arrows,shapes}
\usepackage[colorlinks= true, linkcolor=red, citecolor=blue, urlcolor=black]{hyperref}
\usepackage{cleveref}
\usepackage{marginnote}
\usepackage{diagbox} 
\usepackage{multirow}

\numberwithin{equation}{section}

\theoremstyle{definition}

\newtheorem{theorem}{Theorem}[section]

\newcommand{\mr}{\mathbb{R}}

\usepackage{booktabs}
\usepackage{tikz}
\usepackage{pgfplots}
\usetikzlibrary{calc,intersections}

\usetikzlibrary{shapes.geometric, arrows, positioning}

\tikzset{
	mybox/.style  = {draw, rectangle, minimum width=4cm, minimum height=0.8cm, text centered, text width=4.4cm,   
		font=\normalsize},
	box/.style  = {draw, rectangle, minimum width=2.0cm, minimum height=0.6cm, text centered, text width=3.0cm,   
		font=\normalsize},
	myarrow/.style = {line width=0.2pt, draw=black, -triangle 60, postaction={draw, line width=0.2pt, shorten >=10pt,-}}
}

\tikzstyle{arrow} = [->, >=stealth, -triangle 60]

\allowdisplaybreaks

\makeatletter
\newcommand{\leqnomode}{\tagsleft@true}
\newcommand{\reqnomode}{\tagsleft@false}
\makeatother

\begin{document}

\title{On Underdamped Nesterov's Acceleration\thanks{This work was supported by Grant No.YSBR-034 of CAS and Grant No.12288201 of NSFC.}}


\author[1,2]{Shuo Chen \qquad Bin Shi\thanks{Corresponding author, Email: \url{shibin@lsec.cc.ac.cn}} \qquad Ya-xiang Yuan} 
\affil[1]{State Key Laboratory of Scientific and Engineering Computing, Academy of Mathematics and Systems Science, Chinese Academy of Sciences, Beijing 100190, China}
\affil[2]{University of Chinese Academy of Sciences, Beijing 100049, China}

\date\today

\maketitle

\begin{abstract}
The high-resolution differential equation framework has been proven to be tailor-made for Nesterov's accelerated gradient descent method~(\texttt{NAG}) and its proximal correspondence --- the class of faster iterative shrinkage thresholding algorithms (FISTA). However, the systems of theories is not still complete, since the underdamped case ($r < 2$) has not been included. In this paper, based on the high-resolution differential equation framework, we construct the new Lyapunov functions for the underdamped case, which is motivated by the power of the time $t^{\gamma}$ or the iteration $k^{\gamma}$ in the mixed term. When the momentum parameter $r$ is $2$, the new Lyapunov functions are identical to the previous ones.  These new proofs do not only include the convergence rate of the objective value previously obtained according to the low-resolution differential equation framework but also characterize the convergence rate of the minimal gradient norm square.  All the convergence rates obtained for the underdamped case are continuously dependent on the parameter $r$. In addition, it is observed that the high-resolution differential equation approximately simulates the convergence behavior of~\texttt{NAG}  for the critical case $r=-1$, while the low-resolution differential equation degenerates to the conservative Newton's equation.  The high-resolution differential equation framework also theoretically characterizes the convergence rates, which are consistent with that obtained for the underdamped case with $r=-1$.    
\end{abstract}

%

\section{Introduction}
\label{sec: intro}

Since the advent of the new century, we have witnessed the rapid development of statistical machine learning.  One of the core problems is the unconstrained optimization formulated as
\[
\min_{x\in\mathbb{R}^d}~f(x)
\]
with $f$ being a smooth convex function. Due to its cheap computation and storage, gradient-based optimization has become the workhorse powering recent developments. In the history of gradient-based optimization, a landmark development is Nesterov's accelerated gradient descent method~(\texttt{NAG})
\[
\left\{\begin{aligned}
& x_k = y_{k-1} - s \nabla f(y_{k-1}) \\
& y_{k} = x_{k} + \frac{k-1}{k+r}(x_{k} - x_{k-1}),
\end{aligned} \right.
\]
with any initial $y_0 = x_0 \in \mathbb{R}^d$, which is originally proposed in~\citep{nesterov1983method} with the mathematical tehnique --- \textit{Estimate Sequence} developed by himself to derive the accelerated convergence rate.

However, Nesterov's technique of~\textit{Estimate Sequence} is so algebraically complex that the cause of acceleration is still unclear. Until recently, ~\citet{shi2022understanding} does not propose the high-resolution differential equation framework to successfully lift the veil of the mysterious acceleration phenomenon by the discovery of the~\textit{gradient-correction} term. Meanwhile,~\citet{shi2022understanding} also finds the accelerated convergence phenomenon of the gradient norm minimization. The proof is highly simplified  in~\citep{chen2022gradient}, where the~\textit{implicit-velocity} effect is also found to be an equivalent representation of the~\textit{gradient-correction} term. This manifests that the high-resolution differential equation framework is tailor-made for the~\texttt{NAG}. 

Moreover, the unconstrained optimization that we often meet in practice is the composite optimization as
\[
\min_{x \in \mathbb{R}^d}\Phi(x) := f(x) + g(x),
\]
where $f$ is a smooth convex function and $g$ is a continuous convex function without the assumption to be smooth.  Consequently, with the objective function generalized to be composite, the~\texttt{NAG} becomes the class of so-called faster iterative shrinkage thresholding algorithms (FISTA) as
\[
\left\{\begin{aligned}
& x_k = y_{k-1} - s G_s(y_{k-1}) \\
& y_{k} = x_{k} + \frac{k-1}{k+r}(x_{k} - x_{k-1}),
\end{aligned} \right.
\]
with any initial $y_0 = x_0 \in \mathbb{R}^d$.\footnote{For any $x\in \mathbb{R}^d$, the proximal subgradient $G_s(x)$ is defined in~\Cref{subsec: notation}.} With the key observation to improve the fundamental proximal gradient inequality, the high-resolution differential equation framework is generalized to the composite optimization in~\citep{li2022proximal}.

\subsection{Motivation}
\label{subsec: motivation}

To show our motivation, we first write down the high-resolution differential equation\footnote{Throughout the paper, the high-resolution differential equation refers specifically to the one derived from the~\textit{implicit-velocity} scheme in~\citep{chen2022gradient}. It should be noted that the original high-resolution differential equation is derived from the~\textit{gradient-correction} scheme in~\citep{shi2022understanding}.} as
\begin{equation}
\label{eqn: high-ode}
    \ddot{X}+\frac{r+1}{t}\dot{X} + \left[1+\frac{(r+1)\sqrt{s}}{2t}\right]\nabla f\left(X+\sqrt{s}\dot{X}\right) = 0,
\end{equation}
with any initial $X(0)=x_0$ and $\dot{X}(0) = 0$. From the high-resolution differential equation~\eqref{eqn: high-ode}, we know that the momentum parameter $r$  determines the coefficient of the first-order derivative, or says the velocity, which corresponds to the friction or the damping term in physics. Hence, when the momentum parameter satisfies $r > - 1$, the high-resolution differential equation~\eqref{eqn: high-ode} corresponds to a second-order dissipative potential system. 

\subsubsection{The damped case: $r\in (-1, +\infty)$} 
\label{subsubsec: damped}
For the critically damped case ($r = 2$), it corresponds to the original scheme proposed by~\citep{nesterov1983method} with his technique of ``\textit{Estimate Sequence}'' to obtain the convergence rate of the objective value as
\[
f(y_k) - f(x^\star) \leq O\left(\frac{1}{sk^{2}}\right),
\]
for any step size $0 < s \leq 1/L$. Based on the low-resolution differential equation framework first proposed in~\citep{su2016differential},~\citet{attouch2016rate} discover the faster convergence rate of the objective value for the overdamped case ($r > 2$) as                      
\[
f(y_k) - f(x^\star) \leq o\left(\frac{1}{sk^{2}}\right),
\]
for any step size $0 < s \leq 1/L$. \citet{shi2022understanding} proposes the high-resolution differential equation framework to reproduce the two convergence rates above, and simultaneously finds the accelerated convergence phenomenon for the gradient norm minimization as
\[
\min\limits_{0\leq i\leq k}\|f(y_{i})\|^{2} \leq O\left(\frac{1}{s^2k^{3}}\right),
\] 
for any step size $0 < s \leq 1/L$, which is improved in~\citep{chen2022gradient} as
\[
\min\limits_{0\leq i\leq k}\|f(y_{i})\|^{2} \leq o\left(\frac{1}{s^2k^{3}}\right).
\]
The momentum parameter $r \in (-1,2)$ corresponds to the underdamped case, where the convergence rate of the objective value is derived in~\citep{attouch2019rate} based on the low-resolution differential equation framework as
\[
f(y_k) - f(x^\star) \leq O\left(\left(\frac{1}{sk^2}\right)^{\frac{r+1}{3}}\right).
\]
As a summary to make a comparison, we demonstrate the convergence rates of~\texttt{NAG} with the related mathematical techniques in~\Cref{tab: discrete-case}. 

\begin{table}[htpb!]
\centering
\resizebox{\columnwidth}{!}{
\begin{tabular}{c|cl|cl}
   \toprule
                                                        & $f(y_{k}) - f(x^\star)$                                          &  Math-Techniques                          &$\min\limits_{0\leq i\leq k}\|\nabla f(y_{i})\|^{2}$               & Math-Techniques                               \\
   \midrule
                                                        & \multirow{4}*{$O\left(\frac{1}{sk^{2}}\right)$}     & Estimate Sequence                        &\multirow{4}*{$o\left(\frac{1}{s^2k^{3}}\right)$}                 &                                                           \\
   Critically Damped Case               &                                                                           &\citep{nesterov1983method}            &                                                                                         &  \textbf{High-Resolution ODE}           \\
                  $r=2$                            &                                                                           &\textbf{High-Resolution ODE}           &                                                                                        &  \citep{chen2022gradient}                  \\  
                                                       &                                                                            &\citep{shi2022understanding}          &                                                                                         &                                                           \\
\midrule   
                                                       & \multirow{4}*{$o\left(\frac{1}{sk^{2}}\right)$}       &Low-Resolution ODE                      &\multirow{4}*{$o\left(\frac{1}{s^2k^{3}}\right)$}                  &                                                           \\
           Overdamped Case              &                                                                            &\citep{attouch2016rate}                   &                                                                                          & \textbf{High-Resolution ODE}            \\
           $r \in (2,+\infty)$                  &                                                                           &\textbf{High-Resolution ODE}          &                                                                                          & \citep{chen2022gradient}                   \\  
                                                       &                                                                           &\citep{shi2022understanding}          &                                                                                          &                                                         \\
 \midrule 
                                                       & \multirow{4}*{$O\left(\left(\frac{1}{sk^2}\right)^{\frac{r+1}{3}}\right)$}        &Low-Resolution ODE              & \multirow{4}*{\pmb{?}}                    &                             \\
           Underdamped Case            &                                                                                   &\citep{attouch2019rate}         &                                           &\textbf{High-Resolution ODE} \\
                 $r \in (-1,2)$                  &                                                          &\textbf{High-Resolution ODE}    &                                           &~~$\pmb{?}$                  \\
                                                      &                                                                                   &~~$\pmb{?}$                     &                                           &                             \\
   \bottomrule
    \end{tabular}}
    \caption{The convergence rates of~\texttt{NAG} and the related mathematical techniques used in proof.}
    \label{tab: discrete-case}
\end{table}

From~\Cref{tab: discrete-case}, all the convergence rates for both the critically damped case and the overdamped case can be obtained under the high-resolution differential equation framework instead of the previous mathematical techniques. Naturally, for the underdamped case ($r \in (-1,2)$), we can raise the following two questions under the high-resolution differential equation framework:

\begin{tcolorbox}
\begin{itemize}
\item Can the same convergence rate of the objective value be obtained?
\item How faster does the minimal gradient norm square converge?
\end{itemize}
\end{tcolorbox}

Under the high-resolution differential equation framework, the crucial crux connecting smooth optimization and composite optimization is the improved fundamental proximal gradient inequality, which is observed in~\citep{li2022proximal}. Specifically, for both the critically damped case and the overdamped case, that is, $r \geq 2$, the same convergence rate of the objective values is obtained with $f(x_k) - f(x^\star)$ instead of $f(y_k) - f(x^\star)$, which together with the accelerated rate of the gradient norm minimization is generalized to the proximal case. Similarly, the following question for the underdamped case ($r \in (-1,2)$) is raised as:

\begin{tcolorbox}
\begin{itemize}
\item Can the convergence rates of both the objective value and the gradient norm minimization be generalized to the proximal case?
\end{itemize}
\end{tcolorbox}

\subsubsection{The critical case: $r=-1$}
\label{subsubec: critical}

To understand the convergence behavior for the underdamped case, we take the numerical experiments of~\texttt{NAG} with different momentum parameters $r$ in~\Cref{fig: convergence}. Both the convergence rates of the objective value in~\Cref{subfig: objective-value-nag} and the minimal gradient norm square in~\Cref{subfig: grad-norm-nag} slow down with the decrease of the momentum parameter $r$.
\begin{figure}[htpb!]
\centering
\begin{subfigure}[t]{0.45\linewidth}
\centering
\includegraphics[scale=0.18]{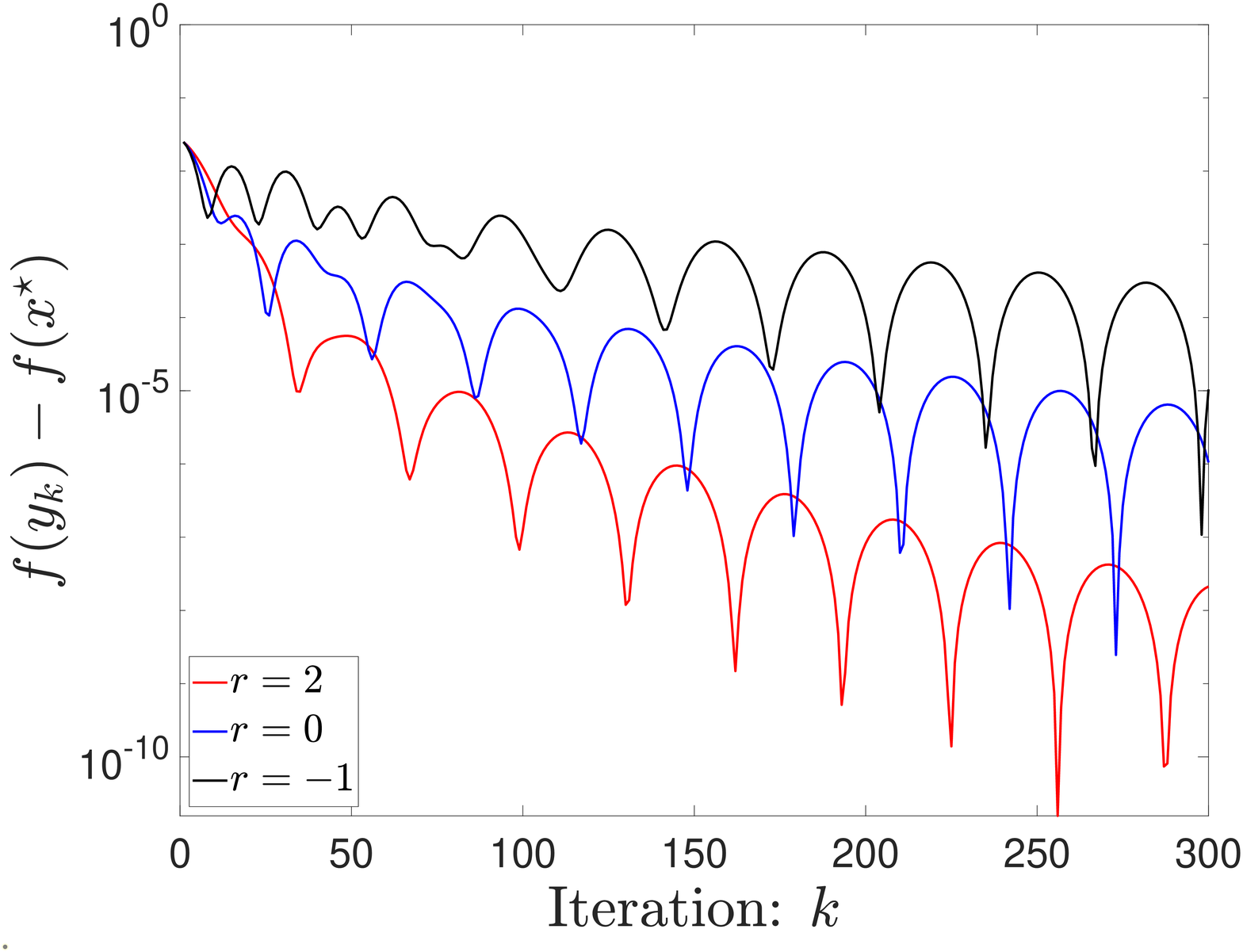}
\caption{Objective Value}
\label{subfig: objective-value-nag}
\end{subfigure}
\begin{subfigure}[t]{0.45\linewidth}
\centering
\includegraphics[scale=0.18]{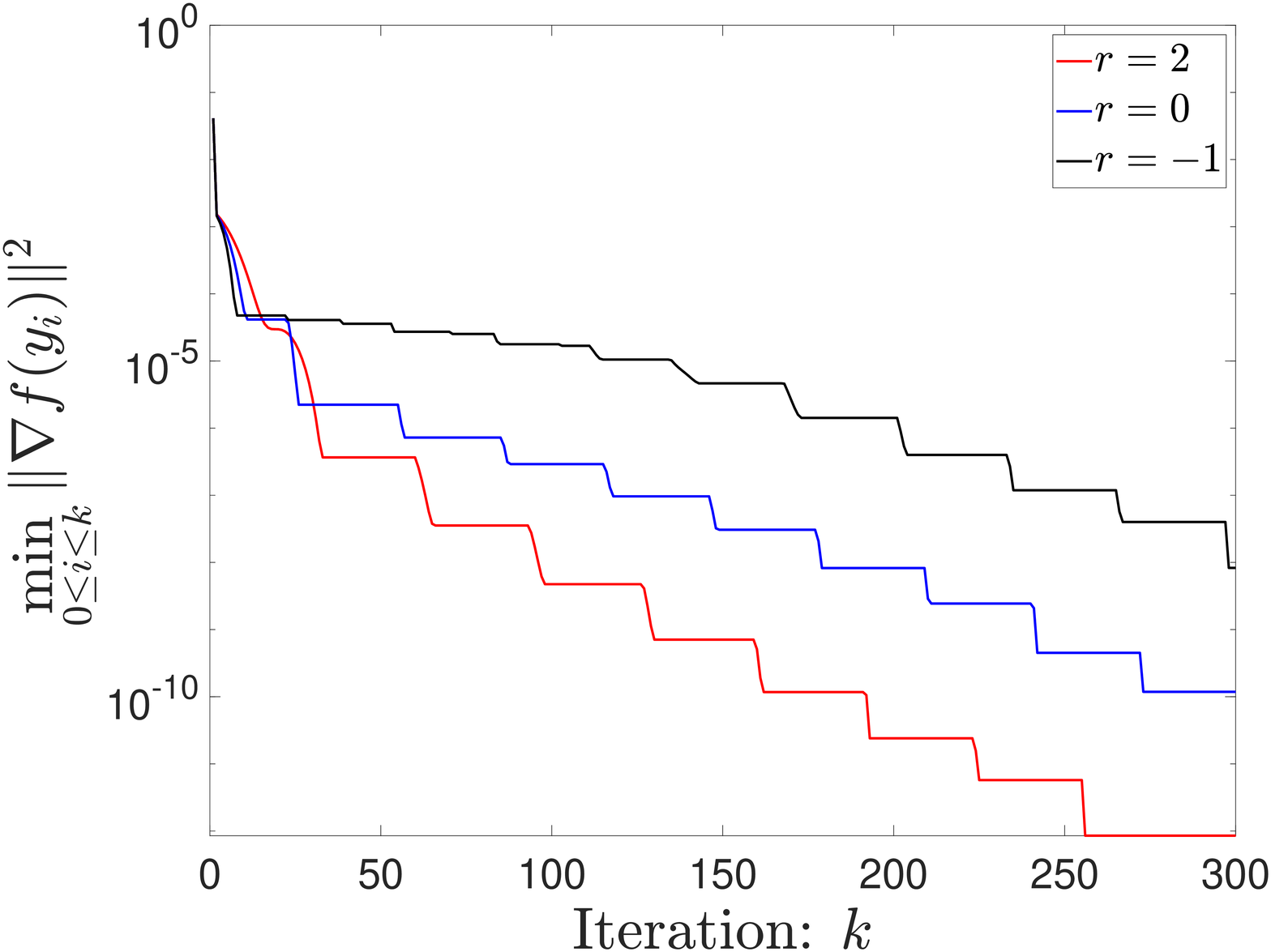}
\caption{Minimal Gradient Norm Square}
\label{subfig: grad-norm-nag}
\end{subfigure}
\caption{Implementing~\texttt{NAG} with the step size $s=0.1$ in the objective function $f(x_1,x_2) = 2\times10^{-2}x_{1}^{2} + 5\times10^{-3}x_{2}^{2}$ from the initial point $\left(1, 1\right)$.}
\label{fig: convergence}
\end{figure}
More interestingly, when the momentum parameter is reduced to the critical case ($r=-1$), the objective value and the minimal gradient norm square still decrease.

Recall the low-resolution differential equation derived in~\citep{su2016differential} as
\begin{equation}
\label{eqn: low-ode}
    \ddot{X}+\frac{r+1}{t}\dot{X} + \nabla f\left(X\right) = 0,
\end{equation}
with any initial $X(0)=x_0$ and $\dot{X}(0) = 0$. When the momentum parameter degenerates to $r=-1$, the low-resolution differential equation~\eqref{eqn: low-ode} is reduced to the standard Newton's equation
\begin{equation}
\label{eqn: low-ode-critical}
\ddot{X} + \nabla f\left(X\right) = 0,
\end{equation}
which is conservative without any friction or damping term to make it converge. In other words, the conservative  Newton's equation~\eqref{eqn: low-ode-critical} is inconsistent with the converging numerical phenomenon of~\texttt{NAG} with $r=-1$, which is shown in~\Cref{fig: ode}. Hence, since the convergence rate derived in~\citep{attouch2019rate} is based on the low-resolution differential equation framework, the analysis cannot be generalized to the critical case ($r=-1$).

\begin{figure}[htpb!]
\centering
\begin{subfigure}[t]{0.45\linewidth}
\centering
\includegraphics[scale=0.18]{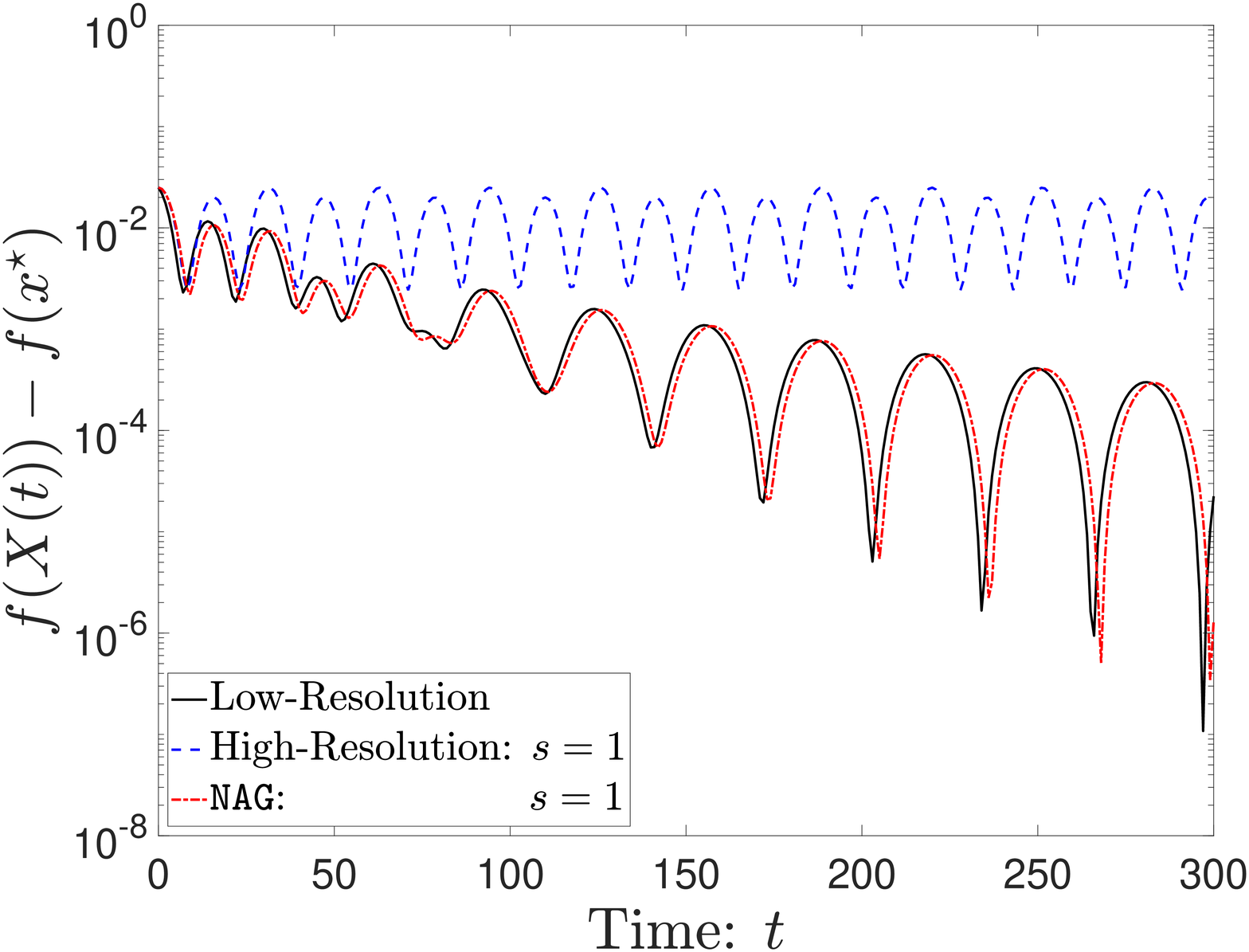}
\caption{Objective Value}
\label{subfig: objective-value-ode}
\end{subfigure}
\begin{subfigure}[t]{0.45\linewidth}
\centering
\includegraphics[scale=0.18]{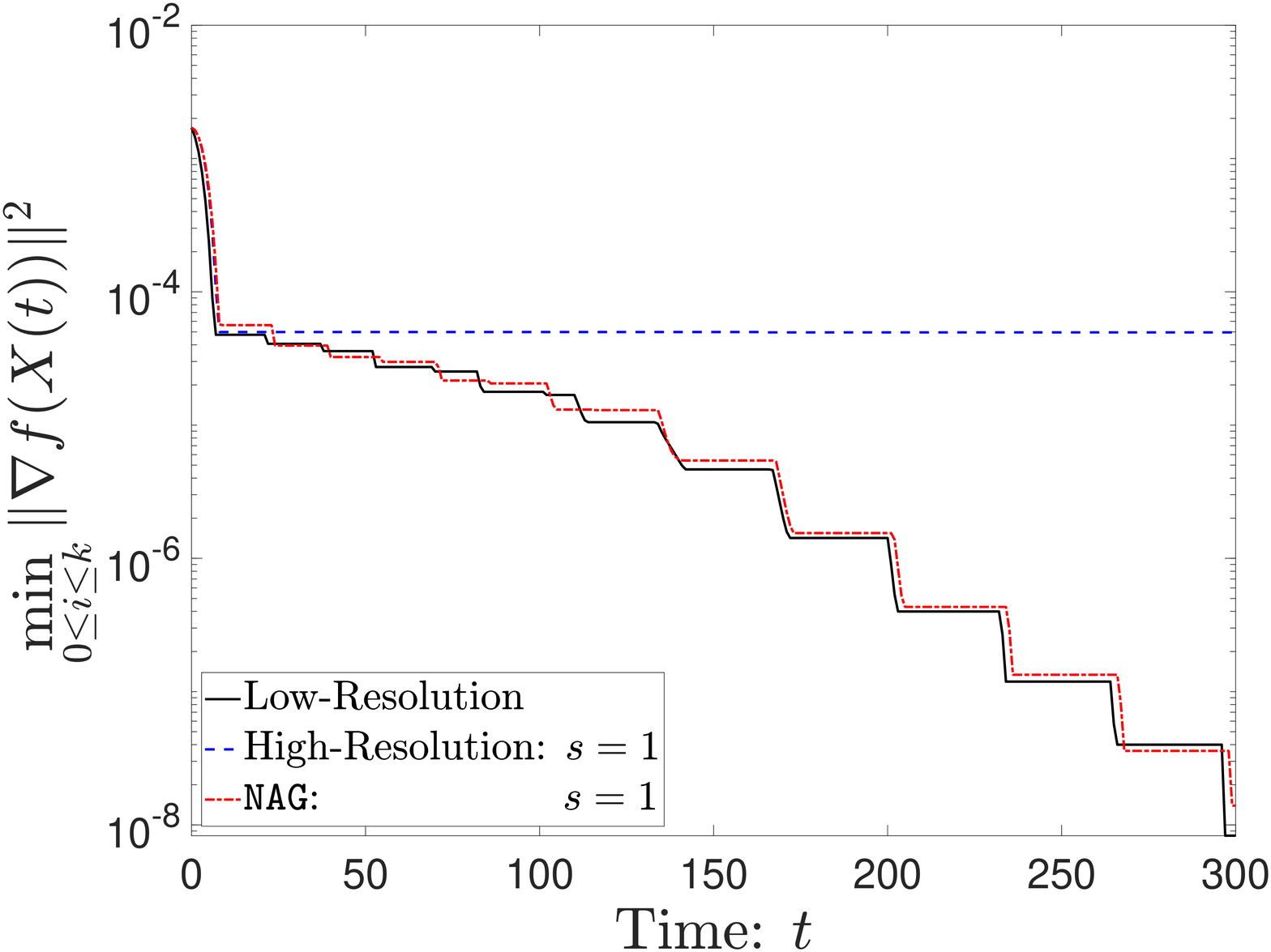}
\caption{Minimal Gradient Norm Square}
\label{subfig: grad-norm-ode}
\end{subfigure}
\caption{The same setting to implement~\texttt{NAG} with the momentum parameter $r=1$ in~\Cref{fig: convergence}. For both the low-resolution and high-resolution differential equations, we use the identification $t=k\sqrt{s}$ between time and iterations for the $x$-axis.}
\label{fig: ode}
\end{figure}

However, we find that the high-resolution differential equation is also very close to~\texttt{NAG} for the critical case ($r=-1$) in~\Cref{fig: ode}. The high-resolution differential equation for the critical case is written as
\begin{equation}
\label{eqn: hig-ode-critical}
\ddot{X} + \nabla f(X + \sqrt{s}\dot{X}) = 0,
\end{equation}
where the gradient term implicitly includes the first-order derivative or the velocity. Taking a simple expansion for the gradient term as 
\[
\nabla f(X + \sqrt{s}\dot{X}) \approx  \sqrt{s}\nabla^2f(X)\dot{X} + \nabla f\left(X\right),
\]
we can find that the high-resolution differential equation~\eqref{eqn: hig-ode-critical} characterizes the damping effect generated by the~\textit{gradient-correction} term or the~\textit{implicit-velocity} term, which exists in the original~\texttt{NAG} with the momentum parameter $r=-1$.

\subsection{Overview of contribution}
\label{sec: contributions}
It has been found in the line works~\citep{shi2022understanding, chen2022gradient, li2022proximal} that the high-resolution differential equation framework with the related mathematical techniques of Lyapunov function and phase-space representation and the~\texttt{NAG} are a perfect match for both the critically damped case and the overdamped case. Motivated by the introduction of the fractional power $k^{(r+1)/3}$ in \citep{attouch2019rate}, we derive the convergence rates by use of the newly constructed Lyapunov function to complement the overdamped case ($r \in [-1,2)$) to the high-resolution differential equation framework in~\Cref{tab: discrete-case}. Meanwhile, we also generalize the convergence rates to the composite function based on the improved fundamental proximal gradient inequality. 

As follows, we briefly describe our main contributions to the convergence rates of~\texttt{NAG} for the underdamped case $r\in[-1,2)$ under the high-resolution differential equation framework, which is based on the gradient Lipschitz inequality and the improved fundamental proximal inequality, respectively.

%

\paragraph{Based on the gradient Lipschitz inequality} For any $x, y \in \mathbb{R}^d$, the $L$-smooth function satisfies the gradient Lipschitz inequality as 
\begin{equation}
\label{eqn: grad-lip}
f(y) - f(x) \leq \left\langle \nabla f(y), y - x \right\rangle - \frac{1}{2L} \|  \nabla f(y) - \nabla f(x) \|^2.
\end{equation}
Taking the continuous high-resolution differential equation as the perspective, we construct the new Lyapunov function via the phase-space representation to derive the convergence rate for the minimal gradient norm square as
\[
\min_{0\leq i\leq k}\|\nabla f(y_{i})\|^{2}\leq o\left(  \frac{1}{s^{\frac{r+4}{3}} k^{\frac{2r+5}{3}} } \right).
\]
for any step size $0 < s \leq 1/L$. Meanwhile, the high-resolution differential equation frameworks also include the previous convergence rate of the objective values 
\[
f(y_{k}) - f(x^\ast) \leq O\left(\left(\frac{1}{sk^2}\right)^{\frac{r+1}{3}}\right).
\]

%

\paragraph{Based on the improved fundamental proximal inequality}
It should be noted that the convergence rates above derived from the gradient Lipschitz inequality are hard to generalize to the composite optimization. In~\citep{li2022proximal}, we find that the bridge connecting the smooth optimization to the composite optimization is the improved fundamental proximal inequality 
as 
\begin{equation}
\label{eqn: prox-grad}
\Phi(y-s G_s(y)) - \Phi(x) \leq \left\langle G_s(y), y - x \right\rangle - \left(s - \frac{Ls^2}{2}\right) \|  G_s(y) \|^2
\end{equation}
for any $0 < s \leq 1/L$. With a similar Lyapunov function, we obtain the convergence rate of the minimal gradient norm square as
\[
\min_{0\leq i\leq k}\|G_s(y_{i})\|^{2}\leq o\left(  \frac{1}{s^{\frac{r+4}{3}} k^{\frac{2r+5}{3}} } \right).
\]
for any step size $0 < s < 1/L$. Meanwhile, the high-resolution differential equation frameworks also include the previous convergence rate for the objective values as
\[
\Phi(x_{k}) - \Phi(x^\ast) \leq O\left(\left(\frac{1}{sk^2}\right)^{\frac{r+1}{3}}\right).
\]


\paragraph{The continuous convergence rates}
As the continuous perspective, we obtain the convergence rates of both the objective value and the minimal gradient norm square from the high-resolution differential equation for the underdamped case ($r \in [-1,2)$). Combined with the results obtained for both the overdamped case and the critically damped case in~\citep{shi2022understanding, chen2022gradient}, we show the convergence rates for all the momentum parameter $r \in [-1, +\infty)$ in~\Cref{tab: continuous-case}, where the bold labels note the convergence rates obtained for the underdamped case in the paper. 

\begin{table}[htpb!]
    \centering
    \begin{tabular}{l|c|c}
   \toprule
                                   & $f(X+\sqrt{s}\Dot{X}) - f(x^\star)$               & $\inf\limits_{t_{0}\leq u\leq t}~\|f(X(u)+\sqrt{s}\dot{X}(u))\|^{2}$ \\
   \midrule
   Overdamped Case                 & \multirow{2}*{$o\left(\frac{1}{t^{2}}\right)$}    & \multirow{2}*{$o\left(\frac{1}{t^{3}}\right)$}                       \\
    ~~~~$r \in (2, +\infty)$           &                                                   &                                                                      \\
   \midrule 
   Critically Damped Case          & \multirow{2}*{$O\left(\frac{1}{t^{2}}\right)$}    & \multirow{2}*{$o\left(\frac{1}{t^{3}}\right)$}                      \\
    ~~~~$r=2$                          &                                                   &                                                                     \\   
   \midrule 
  \textbf{Underdamped Case}        & \multirow{2}*{$\pmb{O\left(\frac{1}{t^{\frac{2(r+1)}{3}}}\right)}$}  & \multirow{2}*{$\pmb{o\left(\frac{1}{t^{\frac{2r+5}{3}}}\right)}$          } \\ 
  ~~~~$\pmb{r \in [-1,2)}$ & &  \\  
   \bottomrule
    \end{tabular}
    \caption{Convergence rates derived from the high-resolution differential equation.}
    \label{tab: continuous-case}
\end{table}

\subsection{Notations}
\label{subsec: notation}

In this paper, we follow the notions of~\citep{nesterov2003introductory} with slight modifications. Let $\mathcal{F}^0$ be the class of continuous convex functions defined on $\mathbb{R}^d$; that is, $g \in \mathcal{F}^0$ if $g(y) \geq g(x) + \langle \nabla g(x), y - x\rangle$ for any $x, y \in \mathbb{R}^d$. The function class $\mathcal{F}^1_{L}$ is the subclass of $\mathcal{F}^0$ with the $L$-smooth gradient; that is, if $f \in \mathcal{F}^1_{L}$, its gradient is $L$-smooth in the sense that 
\[
\| \nabla f(x) - \nabla f(y) \| \leq L \|x - y\|,
\]
for any $x, y \in \mathbb{R}^d$, where $\|\cdot\|$ denotes the standard Euclidean norm and $L > 0$ is the Lipschitz constant. (Note that this implies that $\nabla f$ is also $L'$-Lipschitz for any $L' \geq L$.)  Then, we describe briefly the definition of the proximal operator introduced in~\citep{beck2009fast}. For any $f\in \mathcal{F}_L^1$ and $g \in \mathcal{F}^0$, the proximal operator is defined as
\begin{equation}
\label{eqn: proximal-operator}
    P_s(x) := \mathop{\arg\min}_{z\in\mr^d}\left\{ \frac{1}{2s}\left\| z - \left(x - s\nabla f(x)\right) \right\|^2 + g(z) \right\},
\end{equation}
for any $x \in \mathbb{R}^d$.  With the proximal operator in~\eqref{eqn: proximal-operator}, the proximal subgradient at $x\in \mathbb{R}^d$ of the composite function $\Phi=f+g$ is given by
\begin{equation}
\label{eqn: prox-subgrad}
G_s(x):=\frac{x - P_s(x)}{s}.
\end{equation}



\subsection{Related works and organization}
\label{subsec: related-works}

Since state-of-the-art statistical machine learning springs up in the new century, the research on gradient-based optimization becomes dominant, theoretically and methodologically. It has raised wide concerns for people on the traditional theory of convex optimization, which includes some classical algorithms, such as the vanilla gradient descent, Polyak's heavy ball method~\citep{polyak1964some},~\texttt{NAG}~\citep{nesterov1983method, nesterov2003introductory} and so on. With the rapid development recently, the traditional theory has been extended to composite optimization and image science~\citep{beck2009fast, chambolle2016introduction, beck2017first}. More significantly, state-of-the-art statistical machine learning also gives rise to some
emerging areas, e.g., stochastic optimization~\citep{duchi2011adaptive, kingma2014adam, ghadimi2016accelerated} and nonconvex optimization~\citep{nesterov2018primal, carmon2020first}

At the modern interface between optimization and machine learning, one of the most important features is the continuous model, which introduces the views of differential equations such that the discrete algorithms can be understood and analyzed by use of methods from classical mathematics and physics. A series of research works based on the low-resolution differential equation to model and investigate the~\texttt{NAG} are proposed in~\citep{su2016differential, wibisono2016variational, wilson2021lyapunov, attouch2016rate, attouch2019rate}. Then, some further results are proposed in~\citep{muehlebach2019dynamical, muehlebach2021optimization}. A comprehensive review of these developments is described in~\citep{jordan2018dynamical}. However, the mechanism leading to the acceleration phenomenon generated by the~\texttt{NAG} is not found until the high-resolution differential equation framework recently proposed in~\citep{shi2022understanding}. Currently, stochastic gradient dynamics is introduced to investigate the stochastic optimization algorithms on the nonconvex objective function in~\citep{shi2020learning, shi2021hyperparameters}.

%

The remainder of the paper is organized as follows. In~\Cref{sec: continuous-case}, we construct the Lyapunov function to derive the continuous convergence rates as the perspective based on the high-resolution differential equation. The convergence rates of~\texttt{NAG} is obtained by closing the gap between the continuous model and the discrete algorithms via the phase-space representation in~\Cref{sec: discrete-case}. We generalize the convergence rates to the composite optimization by use of the improved fundamental proximal gradient inequality in~\Cref{sec: proximal}.~\Cref{sec: critical-case} complements the results for the critical case. We conclude the paper with some discussions on new problems and possible future works in~\Cref{sec: conclusion}.

%
%
%


%

\section{High-Resolution differential equation}
\label{sec: continuous-case}

In this section, along with the principled way of constructing Lyapunov functions in~\citep{chen2022revisiting}, we discuss the convergence behavior of the solution to the continuous high-resolution differential equation~\eqref{eqn: high-ode}. For convenience in the following sections of this paper, we use the parameter $\gamma = \frac{r+1}{3} \in (0,1)$ instead of the previous momentum parameter $r$. Hence, the high-resolution differential equation is rewritten down as
\begin{equation}
\label{eqn: high-ode-gamma}
    \ddot{X}+\frac{3\gamma}{t}\dot{X} + \left[1+\frac{3\gamma\sqrt{s}}{2t}\right]\nabla f\left(X+\sqrt{s}\dot{X}\right) = 0,
\end{equation}
with $X(0) = x_0 \in \mathbb{R}^d$ and $\dot{X}(0) = 0$. Then, we proceed to the step of constructing the Lyapunov function.  

\begin{itemize}
\item[\textbf{(I)}] First, we consider the mixed term $\frac{1}{2}\big\|t^\gamma\dot{X}+2\gamma t^{\gamma-1}(X-x^\star)\big\|^2$, where the power $t^{\gamma}$ is inspired by the construction of the Lyapunov function in~\citep{attouch2019rate}. Along the solution to the high-resolution differential equation~\eqref{eqn: high-ode-gamma}, we have
\begin{multline}
\frac{d}{dt} \big[ t^\gamma\dot{X} + 2\gamma t^{\gamma-1}(X-x^\star) \big] \\
                              =  2\gamma(\gamma-1)t^{\gamma-2}(X - x^\star) - t^{\gamma}\left(\frac{1+3\gamma\sqrt{s}}{2t}\right)\nabla f(X+\sqrt{s}\dot{X}) 
\label{eqn: mix-1-con}
\end{multline}
With the equality above~\eqref{eqn: mix-1-con}, we obtain the time derivative of the mixed term as
\begin{align}
\frac{d}{dt} &\bigg( \frac{1}{2}\big\|t^\gamma\dot{X}+2\gamma t^{\gamma-1}(X-x^\star)\big\|^2 \bigg) \nonumber \\
& =  \left\langle t^\gamma\dot{X}+2\gamma t^{\gamma-1}(X-x^\star),  2\gamma(\gamma-1)t^{\gamma-2}(X - x^\star) - t^{\gamma}\left(1+\frac{3\gamma\sqrt{s}}{2t}\right)\nabla f(X+\sqrt{s}\dot{X})  \right\rangle \nonumber  \\
& =  - \underbrace{2\gamma(1-\gamma)t^{2\gamma-2}\langle\dot{X}, X-x^\star\rangle}_{\textbf{I}_{1}} - \underbrace{ 4\gamma^{2}(1-\gamma)t^{2\gamma-3}\|X-x^\star\|^{2}}_{\textbf{I}_{2}} \nonumber  \\
&\mathrel{\phantom{=}} -\underbrace{ t^{\gamma}(t^{\gamma} - 2\gamma\sqrt{s}t^{\gamma-1})\left(1+\frac{3\gamma\sqrt{s}}{2t}\right)\big\langle\dot{X}, \nabla f(X+\sqrt{s}\dot{X})\big\rangle}_{\textbf{I}_{3}} \nonumber  \\
&\mathrel{\phantom{=}}- \underbrace{2\gamma t^{2\gamma-1}\left(1+\frac{3\gamma\sqrt{s}}{2t}\right)\big\langle X+\sqrt{s}\dot{X}-x^\star, \nabla f(X+\sqrt{s}\dot{X})\big\rangle}_{\textbf{I}_{4}}.  \label{eqn: mix-2-con}
\end{align}
From the equality~\eqref{eqn: mix-2-con}, we find that both the terms, $\textbf{I}_2$ and  $\textbf{I}_4$, are no less than zero, so it is enough for us to adjust the coefficients of the distance square and the potential function to eradicate the term $\textbf{I}_1$ and the term $\textbf{I}_3$, respectively.

\item[\textbf{(II)}] Then, we consider the distance square $\| X - x^\star\|^2$. To cancel out the term $\textbf{I}_1$, we need to put the coefficient  $\gamma(1-\gamma)t^{2\gamma-2}$ before the distance square. Hence, we have
\begin{multline}
\label{eqn: dist-1-con}
\frac{d}{dt} \left( \gamma(1-\gamma)t^{2\gamma-2}\| X - x^\star\|^2 \right)  \\=  \underbrace{ 2\gamma(1-\gamma)t^{2\gamma-2}\langle\dot{X}, X-x^\star\rangle}_{\textbf{II}_{1}} - \underbrace{2\gamma(1-\gamma)^2t^{2\gamma-3}\|X-x^\star\|^2}_{\textbf{II}_{2}}.
\end{multline}
Compared with the two inequalities above~\eqref{eqn: mix-2-con} and~\eqref{eqn: dist-1-con}, we obtain that the two terms, $\textbf{I}_1$ and $\textbf{II}_1$, cancel each other out, that is, $\textbf{II}_1 - \textbf{I}_1 = 0$.

\item[(\textbf{III})] Finally, we consider the potential function $f(X + \sqrt{s}\dot{X}) - f(x^\star)$. Let $g(t)$ be the coefficient of the potential function.   Then, along the solution to the high-resolution differential equation~\eqref{eqn: high-ode-gamma}, we obtain the derivative of the potential function as
\begin{align}
\frac{d}{dt}\big[g(t)&\big( f(X + \sqrt{s}\dot{X}) - f(x^\star) \big)\big]   \nonumber \\
                                 & =  \underbrace{g(t)\left(1 - \frac{3\gamma\sqrt{s}}{t}\right) \big\langle \dot{X}, \nabla f(X+\sqrt{s}\dot{X}) \big\rangle}_{\textbf{III}_1} \nonumber \\ & \mathrel{\phantom{=}} - \underbrace{\sqrt{s}g(t)\left(1+\frac{3\gamma\sqrt{s}}{2t}\right) \big\|\nabla f(X+\sqrt{s}\dot{X})\big\|^2}_{\textbf{III}_2} + \underbrace{g'(t)\left( f(X + \sqrt{s}\dot{X}) - f(x^\star) \right) }_{\textbf{III}_3}   
\label{eqn: pot-1-con}
\end{align}
To make the term $\textbf{III}_1$ eradicating the term $\textbf{I}_3$, that is, $\textbf{III}_1 - \textbf{I}_3 = 0$, we set the coefficient as
\begin{equation}
\label{eqn: coeff-pot}
g(t) = \frac{t+3\gamma\sqrt{s}/ 2}{t - 3\gamma\sqrt{s}} \cdot  t^{\gamma}(t^{\gamma} - 2\gamma\sqrt{s}t^{\gamma-1}).
\end{equation}
With the expression of the coefficient~\eqref{eqn: coeff-pot}, we calculate the derivative $g'(t)$ as
\[
g'(t)= 2\gamma t^{2\gamma-1}\left(1+\frac{3\gamma\sqrt{s}}{2t}\right)  + \frac{\left[(4\gamma-5)t^2 + 6(2-\gamma)\gamma\sqrt{s}t-18\gamma^2(1+\gamma)s\right]\cdot \gamma \sqrt{s}t^{2(\gamma-1)} } {2\left(t-3\gamma\sqrt{s}\right)^2}.
\]
With some basic operations, we know $4\gamma - 5 < 0$ from $\gamma \in (0,1)$.  Hence, there exists $t_1(\gamma)$ only dependent on the paramere $\gamma$ such that the derivative of the coefficient satisfies
\[
g'(t) \geq  2\gamma t^{2\gamma-1}\left(1+\frac{3\gamma\sqrt{s}}{2t}\right)
\]
for any $t \geq t_{1}(\gamma)$.
\end{itemize}
 

Summing up the time derivatives of the mixed term~\eqref{eqn: mix-2-con}, the distance square~\eqref{eqn: dist-1-con} and the potential function~\eqref{eqn: pot-1-con}, we construct the new Lyapunov function as
\begin{multline}
\label{eqn: lya-con}
\mathcal{E}(t) = \frac{t^{\gamma}(t^{\gamma}-2\gamma\sqrt{s}t^{\gamma-1})}{t-3\gamma\sqrt{s}}\left(t+\frac{3\gamma\sqrt{s}}{2}\right)\left(f(X+\sqrt{s}\dot{X}) - f(x^\star)\right) \\
+ \frac{1}{2}\|t^\gamma\dot{X}+2\gamma t^{\gamma-1}(X-x^\star)\|^2 + \gamma(1 - \gamma)t^{2(\gamma-1)}\|X - x^\star\|^2,
\end{multline}
for any $\gamma \in (0,1)$. When the parameter satisfies $\gamma=1$, that is, the critically damped case with $r=2$, the Lyapunov function above~\eqref{eqn: lya-con} is reduced to the one constructed in~\citep[(4.3)]{chen2022gradient}. With the Lyapunov function above~\eqref{eqn: lya-con}, we derive the convergence rates of the objective value and the gradient norm square, which is rigorously described by the following theorem.

\begin{theorem}
\label{thm: con}
Let $f\in \mathcal{F}_1^L(\mathbb{R}^d)$, then there exists $t_{0} := \max\{3\gamma\sqrt{s}+1, t_{1}(\gamma)\}$ only dependent on $\gamma\in(0,1)$ such that the solution $X=X(t)$ to the high-resolution differential equation~\eqref{eqn: high-ode} obeys
\begin{equation}
\label{eqn: con-gradnorm}
    \lim_{t\to\infty}t^{2\gamma+1}\|\nabla f(X+\sqrt{s}\Dot{X})\|^2=0
\end{equation}
and
\begin{equation}
\label{eqn: con-objvalue}
    f(X+\sqrt{s}\dot{X}) - f(x^\star) \leq \frac{\mathcal{E}(t_{0})}{t^{\gamma}(t^{\gamma} - 2\gamma \sqrt{s}t^{\gamma-1})}
\end{equation}
for all $t\geq t_{0}$.
\end{theorem}
\begin{proof}[Proof of~\Cref{thm: con}]
From the time derivatives of the mixed term~\eqref{eqn: mix-2-con} and the distance square~\eqref{eqn: dist-1-con}, we can obtain the following equality as
\[
-\textbf{I}_2 - \textbf{II}_2 = -2\gamma(1 - \gamma^2) \|X - x^\star\|^2
\] 
Combined with the time derivative of the potential function~\eqref{eqn: pot-1-con}, we obtain the time derivative of the Lyapunov function~\eqref{eqn: lya-con}
\begin{multline}
\frac{d\mathcal{E}(t)}{dt} =  -2\gamma(1 - \gamma^2) \|X - x^\star\|^2 - \frac{(t+3\gamma\sqrt{s}/ 2)^2}{t(t - 3\gamma\sqrt{s})} \cdot  \sqrt{s}t^{\gamma}(t^{\gamma} - 2\gamma\sqrt{s}t^{\gamma-1}) \big\|\nabla f(X+\sqrt{s}\dot{X})\big\|^2 \\
                             - 2\gamma t^{2\gamma-1}\left(1+\frac{3\gamma\sqrt{s}}{2t}\right) \left[\big\langle X+\sqrt{s}\dot{X}-x^\star, \nabla f(X+\sqrt{s}\dot{X})\big\rangle  - \left( f(X + \sqrt{s}\dot{X}) - f(x^\star) \right) \right]. \label{eqn: con-proof-1}
\end{multline}
With the gradient Lipschtiz inequality~\eqref{eqn: grad-lip}, the the estimate of the time derivative~\eqref{eqn: con-proof-1} can be loosen as
\[
\frac{d\mathcal{E}(t)}{dt} \leq -\sqrt{s}t^{\gamma}(t^{\gamma} - \gamma\sqrt{s}t^{\gamma-1}) \big\|\nabla f(X+\sqrt{s}\dot{X})\big\|^2.  
\]
which implies the Lyapunov function $\mathcal{E}(t)$ given in~\eqref{eqn: lya-con} does not increase with the time $t \geq t_0$, leading to \eqref{eqn: con-objvalue}. Meanwhile, we can obtain the following inequality as
\[
\frac{\sqrt{s}}{2}\int_{t_{0}}^{t} u^{2\gamma} \big\|\nabla f(X+\sqrt{s}\dot{X})\big\|^2 du \leq  \int_{t_{0}}^{t} \sqrt{s}u^{\gamma}(u^{\gamma} - \gamma\sqrt{s}u^{\gamma-1}) \big\|\nabla f(X+\sqrt{s}\dot{X})\big\|^2 du \leq \mathcal{E}(t_0)
\]
for any time $t \geq t_0$, which implies that
\[
0 \leq \frac{2^{2\gamma+1} - 1}{(2\gamma+1)2^{2\gamma+1}} \lim_{t\rightarrow \infty}\left( t^{2\gamma+1}\min_{t_0 \leq u \leq t} \big\|\nabla f(X+\sqrt{s}\dot{X}) \big\|^2\right) \leq \lim_{t\rightarrow \infty} \int_{\frac{t}{2}}^{t} u^{2\gamma} \big\|\nabla f(X+\sqrt{s}\dot{X})\big\|^2 du = 0.
\]
It follows that \eqref{eqn: con-gradnorm} holds. The proof is complete. 
\end{proof}

\section{Underdamped Nesterov's acceleration}
\label{sec: discrete-case}
In this section, we carry out the transition to the discrete convergence rates of~\texttt{NAG} via the phase-space representation, which is based on the continuous perspective from the high-resolution differential equation.  By introducing the new sequence $\{v_{k}\}_{k=0}^{\infty}$ with $v_{k} = (x_{k} - x_{k-1})/\sqrt{s}$, we show the phase-space representation for the discrete~\texttt{NAG} as
\begin{equation}
\label{eqn: phase-x}
\left\{\begin{aligned}
& x_{k+1} - x_{k} = \sqrt{s}v_{k+1}, \\
& v_{k+1} - v_{k} = -\frac{3\gamma}{k+3\gamma-1}\cdot v_k - \sqrt{s}\nabla f\left(y_{k}\right),
\end{aligned} \right.
\end{equation}
with any initial $x_{0}\in\mathbb{R}^d$ and $v_0=0$, where the variable $y_{k}$ satisfies the following relation as
\begin{equation}
\label{eqn: nag-2-yx}
    y_k = x_{k} + \frac{k-1}{k+3\gamma-1}(x_{k}-x_{k-1}) = x_k + \frac{k-1}{k+3\gamma-1}\cdot\sqrt{s}v_k.
\end{equation}
Then, we proceed to the step of constructing the Lyapunov function.

\begin{itemize}
\item[(\textbf{I})] Similarly, we first consider the mixed term $\frac12\| (k-1)^{\gamma}s^{\gamma/2}v_{k} + 2\gamma k^{\gamma-1}s^{(\gamma-1)/2}(x_{k} - x^\star) \|^2$, which is consistent with the continuous case in the previous section. With the phase-space representation~\eqref{eqn: phase-x} and the relation among the variables~\eqref{eqn: nag-2-yx}, we have
\begin{align}
     &\big[k^{\gamma}s^{\gamma/2}v_{k+1} + 2\gamma(k+1)^{\gamma-1}s^{(\gamma-1)/2}(x_{k+1} - x^\star)\big] \nonumber \\
     &\mathrel{\phantom{\big[k^{\gamma}s^{\gamma/2}v_{k+1} + 2\gamma(k+1)^{\gamma-1}s^{(\gamma-1)/2}}} - \big[(k-1)^{\gamma}s^{\gamma/2}v_{k} + 2\gamma k^{\gamma-1}s^{(\gamma-1)/2}(x_{k} - x^\star)\big]  \nonumber \\
 & \mathrel{\phantom{=}}  =    s^{\gamma/2}(k - 1) \cdot \frac{k^{\gamma} + 2\gamma(k+1)^{\gamma-1} - (k-1)^{\gamma-1}(k+3\gamma-1) }{k + 3\gamma - 1} \cdot v_{k} \nonumber \\
    & \mathrel{\phantom{==}}  - s^{(\gamma+1)/2}\left[ k^{\gamma} + 2\gamma(k+1)^{\gamma-1}\right] \nabla f(y_{k}) + 2\gamma s^{(\gamma-1)/2}\left[(k+1)^{\gamma-1} - k^{\gamma-1} \right](x_{k} - x^\star) \label{eqn: nag-1-mix}
\end{align}
and
\begin{align}
     &\big[k^{\gamma}s^{\gamma/2}v_{k+1} + 2\gamma(k+1)^{\gamma-1}s^{(\gamma-1)/2}(x_{k+1} - x^\star)\big] \\
     &\mathrel{\phantom{\big[k^{\gamma}s^{\gamma/2}v_{k+1} + 2\gamma(k+1)^{\gamma-1}s^{(\gamma-1)/2}}} + \big[(k-1)^{\gamma}s^{\gamma/2}v_{k} + 2\gamma k^{\gamma-1}s^{(\gamma-1)/2}(x_{k} - x^\star)\big]  \nonumber \\
 & \mathrel{\phantom{=}}  =    s^{\gamma/2}(k - 1) \cdot \frac{k^{\gamma} + 2\gamma(k+1)^{\gamma-1} + (k-1)^{\gamma-1}(k+3\gamma-1) }{k + 3\gamma - 1} \cdot v_{k} \nonumber \\
    & \mathrel{\phantom{==}}  - s^{(\gamma+1)/2}\left[ k^{\gamma} + 2\gamma(k+1)^{\gamma-1}\right] \nabla f(y_{k}) + 2\gamma s^{(\gamma-1)/2}\left[(k+1)^{\gamma-1} + k^{\gamma-1} \right](x_{k} - x^\star) \label{eqn: nag-2-mix}
\end{align}
%
With~\eqref{eqn: nag-1-mix} and~\eqref{eqn: nag-2-mix}, we obtain the iterative difference of the mixed term as
\begin{align}
   \frac12 \big\|k^{\gamma}&s^{\gamma/2}v_{k+1} + 2\gamma (k+1)^{\gamma-1}s^{(\gamma-1)/2}(x_{k+1} - x^\star)\big\|^2 \nonumber \\
   &  \mathrel{\phantom{v_{k+1} + 2\gamma (k+1)^{\gamma-1}s^{(\gamma-1)}}} + \frac12 \big\|(k-1)^{\gamma}s^{\gamma/2}v_{k} + 2\gamma k^{\gamma-1}s^{(\gamma-1)/2}(x_{k} - x^\star)\big\|^2 \nonumber \\
=  & \frac{s^{\gamma}}{2}A_{k}\|v_{k}\|^2 + \frac{s^{\gamma-1}}{2}B_{k}\|x_{k} - x^\star\|^2 + \frac{s^{\gamma+1}}{2}C_{k}\|\nabla f(y_{k})\|^2  \nonumber \\
   & + \frac{s^{\gamma+1/2}}{2}D_{k}\langle\nabla f(y_{k}), v_{k}\rangle  + \frac{s^{\gamma}}{2}E_{k}\langle\nabla f(y_{k}), x_{k} - x^\star\rangle  + \frac{s^{\gamma-1/2}}{2}F_{k}\langle v_{k}, x_{k} - x^\star\rangle,  \label{eqn: nag-3-mix}
\end{align}
where the coefficients are represented respectively as
\begin{align*}
& A_{k} = \left(\frac{k-1}{k+3\gamma-1}\right)^2 \cdot \big\{[k^{\gamma} + 2\gamma(k+1)^{\gamma-1}]^2 - (k-1)^{2(\gamma-1)}(k+3\gamma-1)^2 \big\}, \\
& B_{k} = 4\gamma^2  \big[(k+1)^{2(\gamma-1)} - k^{2(\gamma-1)} \big],                              \\
& C_{k} = \big[k^{\gamma} + 2\gamma(k+1)^{\gamma-1}\big]^2,                                            \\
& D_{k} = - \frac{2 (k - 1) \big[k^{\gamma} + 2\gamma(k+1)^{\gamma-1}\big]^2}{k + 3\gamma - 1 },   \\
& E_{k} = - 4\gamma(k+1)^{\gamma-1} \big[k^{\gamma} + 2\gamma(k+1)^{\gamma-1}\big], \\
& F_{k} = \frac{4\gamma (k - 1)}{k + 3\gamma -1} \cdot \big\{ (k+1)^{\gamma-1}[k^{\gamma} + 2\gamma(k+1)^{\gamma-1}] - k^{\gamma-1}(k-1)^{\gamma-1}(k + 3\gamma-1 )\big\}.
\end{align*}

With the relation among the variables~\eqref{eqn: nag-2-yx}, the iterative difference of the mixed term~\eqref{eqn: nag-3-mix} can be reformulated as
\begin{align}
   \frac12 \big\|k^{\gamma}&s^{\gamma/2}v_{k+1} + 2\gamma (k+1)^{\gamma-1}s^{(\gamma-1)/2}(x_{k+1} - x^\star)\big\|^2 \nonumber \\
   &  \mathrel{\phantom{v_{k+1} + 2\gamma (k+1)^{\gamma-1}s^{(\gamma-1)}}} + \frac12 \big\|(k-1)^{\gamma}s^{\gamma/2}v_{k} + 2\gamma k^{\gamma-1}s^{(\gamma-1)/2}(x_{k} - x^\star)\big\|^2 \nonumber  \\
=  & \underbrace{\frac{s^{\gamma}}{2}A_{k}\|v_{k}\|^2}_{\textbf{I}_1} +  \underbrace{\frac{s^{\gamma-1}}{2}B_{k}\|x_{k} - x^\star\|^2}_{\textbf{I}_2}  + \underbrace{\frac{s^{\gamma+1}}{2}C_{k}\|\nabla f(y_{k})\|^2}_{\textbf{I}_3}   \nonumber \\
   & + \underbrace{\frac{s^{\gamma}}{2}G_k\langle\nabla f(y_{k}), y_{k} - x_{k}\rangle}_{\textbf{I}_4} +\underbrace{ \frac{s^{\gamma}}{2}E_{k}\langle\nabla f(y_{k}), y_{k} - x^\star\rangle }_{\textbf{I}_5} + \underbrace{ \frac{s^{\gamma-1/2}}{2}F_{k}\langle v_{k}, x_{k} - x^\star\rangle}_{\textbf{I}_6},  \label{eqn: nag-4-mix}
\end{align}
where the coefficient $G_{k}$ satisfies
\[
G_k = \frac{k+3\gamma-1}{k-1} \cdot  D_{k} -  E_{k}  = -2k^{\gamma}\left[k^{\gamma} + 2\gamma(k+1)^{\gamma-1}\right] \leq 0.    
\]
With some basic calculus, it is easy to infer that the coefficient $G_k$ decreases for any $\gamma \in (0,1)$. In addition, we also observe that the coefficients satisfy the following relation as
\begin{equation}
\label{eqn: mix-ceg}
2C_{k} + E_{k} + G_{k} = 0.
\end{equation}

\item[(\textbf{II})] To eradicate the term $\textbf{I}_6$, we consider the distance square $\|x_{k-1} - x^\star\|^2$ with an undetermined coefficient $\frac{s^{\gamma-1}H_{k}}{2}$. With the phase-space representation~\eqref{eqn: phase-x}, the difference is calculated as
\begin{align}
    &  \frac{s^{\gamma-1}H_{k+1}}{2}\|x_{k} - x^\star\|^2 - \frac{s^{\gamma-1}H_k}{2}\|x_{k-1} - x^\star\|^2  \nonumber \\
  =  & \frac{s^{\gamma-1}H_{k+1}}{2}\left\langle x_{k} - x_{k-1}, x_{k} + x_{k-1} - 2x^\star \right\rangle - \frac{s^{\gamma-1}(H_{k} - H_{k+1}) }{2} \|x_{k-1} - x^\star\|^2  \nonumber \\
  =  & \frac{s^{\gamma-1}H_{k+1}}{2}\left\langle \sqrt{s}v_{k},2(x_{k} - x^\star) - \sqrt{s}v_{k} \right\rangle - \frac{s^{\gamma-1}(H_{k} - H_{k+1}) }{2} \|x_{k-1} - x^\star\|^2  \nonumber \\
  =  & \underbrace{s^{\gamma-\frac12} H_{k+1} \left\langle v_{k}, x_{k} - x^\star \right\rangle}_{\mathbf{II}_1} - \underbrace{\frac{s^{\gamma}H_{k+1}}{2} \|v_{k}\|^2}_{\mathbf{II}_2} - \underbrace{ \frac{s^{\gamma-1}(H_{k} - H_{k+1}) }{2} \|x_{k-1} - x^\star\|^2}_{\mathbf{II}_3}. \label{eqn: nag-1-dist}
\end{align}
With the iterative difference of the distance square~\eqref{eqn: nag-1-dist},  the desired result $\textbf{II}_1 + \mathbf{I}_6 = 0$ requires the coefficients satisfies 
\begin{align*}
H_{k} & = - \frac{F_{k-1}}{2}\\
      & = - \frac{2\gamma (k-2)}{k + 3\gamma -2} \cdot \big\{ k^{\gamma-1}[(k-1)^{\gamma} + 2\gamma k^{\gamma-1}] - (k-1)^{\gamma-1}(k-2)^{\gamma-1}(k + 3\gamma - 2 )\big\} \\
      & = \frac{2\gamma (k-2)}{k + 3\gamma -2} \cdot \big\{ (1 - \gamma ) k^{2\gamma - 2} + O\left(k^{2\gamma-3}\right) \big\}.
\end{align*}
Hence, we know that there exists $K_1(\gamma) \geq 2$ such that the coefficient $H_{k}$ is no less than zero and does not increase for any $k \geq K_1(\gamma)$. In other words, when the iteration number is large enough, that is, $k \geq K_2(\gamma)$, the coefficient satisfies $H_{k} = - F_{k-1}/2$ such that both the terms $\textbf{II}_2$ and $\textbf{II}_{3}$ are non-negative.

\item[(\textbf{III})] To eradicate the term $\textbf{I}_4$, we consider the potential function $f(y_{k-1}) - f(x^\star)$ with the coefficient $- \frac{s^{\gamma}G_{k}}{2}$. Then, we obtain the iterative difference of the potential function as 
\begin{multline}
    - \frac{s^{\gamma}G_{k+1}}{2} (f(y_{k}) - f(x^\star)) + \frac{s^{\gamma}G_{k}}{2} (f(y_{k-1}) - f(x^\star)) \\
 =   - \frac{s^{\gamma}G_{k}}{2} (f(y_{k}) - f(y_{k-1})) - \frac{s^{\gamma}(G_{k+1} - G_{k})}{2} (f(y_{k}) - f(x^\star)). \label{eqn: nag-1-pot}
\end{multline}
Putting the iterative sequence $y_{k+1}$ and $y_{k}$ into the gradient Lipschtiz inequality~\eqref{eqn: grad-lip}, we obtain the following inequality as
\begin{equation}
 f(y_{k}) - f(y_{k-1}) \leq\left\langle \nabla f(y_{k}), y_{k} - y_{k-1} \right\rangle - \frac{1}{2L}\left\|\nabla f(y_{k}) - \nabla f(y_{k-1})\right\|^2.\label{eqn: grad-lip-1}
\end{equation}
With the inequality~\eqref{eqn: grad-lip-1}, we proceed to estimate the iterative difference of the potential function~\eqref{eqn: nag-1-pot} as
\begin{align}
    - \frac{s^{\gamma}G_{k+1}}{2} (f(y_{k}) - &f(x^\star)) + \frac{s^{\gamma}G_{k}}{2} (f(y_{k-1}) - f(x^\star))\nonumber \\
\leq  & - \underbrace{\frac{s^{\gamma}G_{k}}{2} \left(\left\langle \nabla f(y_{k}), y_{k} - y_{k-1} \right\rangle - \frac{1}{2L}\left\|\nabla f(y_{k}) - \nabla f(y_{k-1})\right\|^2\right)}_{\textbf{III}_1} \nonumber  \\
 &- \underbrace{\frac{s^{\gamma}(G_{k+1} - G_{k})}{2} \left(f(y_{k}) - f(x^\star)\right)}_{\textbf{III}_2}. \label{eqn: nag-2-pot}
\end{align}
Since the step size satisfies $0 < s \leq 1/L$, we utilize the gradient step of~\texttt{NAG} to obtain the following inequality as
\begin{align}
\textbf{I}_3 + \textbf{I}_4  - \textbf{III}_1 \leq & \frac{s^{\gamma+1}}{2}C_{k}\|\nabla f(y_{k})\|^2   \nonumber \\ 
                                                   & - \frac{s^{\gamma}G_{k}}{2} \left( - s \left\langle \nabla f(y_{k}),  \nabla f(y_{k-1}) \right\rangle - \frac{s}{2}\left\|\nabla f(y_{k}) - \nabla f(y_{k-1})\right\|^2\right) \nonumber \\ 
                                                =  &  - \frac{s^{\gamma+1}E_{k}}{4}\|\nabla f(y_{k})\|^2 + \frac{s^{\gamma+1}G_{k}}{4} \|\nabla f(y_{k-1})\|^2,  \label{eqn: nag-3-pot}
\end{align}
where the last equality follows the relation among the coefficients~\eqref{eqn: mix-ceg}.
\end{itemize}

Summing up the iterative derivatives of the mixed term~\eqref{eqn: nag-4-mix}, the distance square~\eqref{eqn: nag-1-dist} and the potential function~\eqref{eqn: nag-2-pot}, we construct the new Lyapunov function as
\begin{align}
\mathcal{E}(k)  = & s^{\gamma}k^{\gamma}\left[k^{\gamma} + 2\gamma(k+1)^{\gamma-1}\right] \left( f(y_{k-1}) - f(x^\star) \right)  \nonumber \\
                  & + \frac12 \big\|(k-1)^{\gamma}s^{\gamma/2}v_{k} + 2\gamma k^{\gamma-1}s^{(\gamma-1)/2}(x_{k} - x^\star)\big\|^2 \nonumber \\
                  & + 2 \gamma s^{\gamma-1} \left\{ \frac{(k-2)k^{\gamma-1}[(k-1)^{\gamma} + 2\gamma k^{\gamma-1}]}{k + 3\gamma -2} - (k-1)^{\gamma-1}(k-2)^{\gamma} \right\} \|x_{k-1} - x^\star\|^2. \label{eqn: lyapunov-nag}
\end{align}
for any $\gamma \in (0,1)$. When the parameter satisfies $\gamma=1$, that is, the critically damped case with $r=2$, the Lyapunov function above~\eqref{eqn: lyapunov-nag} is reduced to the one constructed in~\citep[(4.9)]{chen2022gradient}. With the Lyapunov function above~\eqref{eqn: lyapunov-nag}, we derive the convergence rates of the objective value and the gradient norm square, which is rigorously described by the following theorem. 

%

\begin{theorem}
\label{thm: conv-rate-nag}
Let $f \in \mathcal{F}_{L}^{1}(\mathbb{R}^d)$. There exists $K_{0}:=K_0(\gamma)$ only depending on $\gamma \in(-1,2)$ such that the iterative sequence $\{y_{k}\}_{k=0}^{\infty}$ generated by \texttt{NAG} with any step size $0 < s \leq 1/L$ obeys
\begin{equation}
\label{eqn: nag-gradnorm}
    \lim_{k\to\infty}\left[s^{\gamma+1}k^{2\gamma+1}\min_{0\leq i\leq k}\|\nabla f(y_{i})\|^{2}\right] = 0,
\end{equation}
and 
\begin{equation}
\label{eqn: nag-objvalue}
f(y_{k}) - f(x^\star)\leq\frac{\mathcal{E}(K_{0})}{s^{\gamma}(k+1)^{2\gamma}}
\end{equation}
for any $k \geq K_0$. Especially, let the step size be set as $s=1/L$, then we have
\[
\lim_{k\to\infty}\left[\frac{k^{2\gamma+1}}{L^{\gamma+1}}\min_{0\leq i\leq k}\|\nabla f(y_{i})\|^{2}\right] = 0\quad \text{and} \quad f(y_{k}) - f(x^\star)\leq\frac{L^{\gamma}\mathcal{E}(K_{0})}{(k+1)^{2\gamma}}
\]
for any $k \geq K_0$.
\end{theorem}
\begin{proof}[Proof of~\Cref{thm: conv-rate-nag}]
Taking the basic binomial expansion, we obtain the following expression as
\[
(k+1)^{\gamma} + 2\gamma(k+2)^{\gamma-1} - k^{\gamma-1}(k+3\gamma) = - \frac{9\gamma(1-\gamma)}{2} \cdot k^{\gamma-2} + O\left(  k^{\gamma-3} \right).
\]
Hence, we know that there exists $K_2(\gamma)$ only dependent on the parameter $\gamma$ such that 
\[
A_{k} =  \left(\frac{k}{k+3\gamma}\right)^2 \cdot \big\{[(k+1)^{\gamma} + 2\gamma(k+2)^{\gamma-1}]^2 - k^{2(\gamma-1)}(k+3\gamma)^2 \big\} \leq 0
\]
for any $k \geq K_2(\gamma)$. In addition, we also know $B_{k} = 4\gamma^2  \big[(k+2)^{2(\gamma-1)} - (k+1)^{2(\gamma-1)} \big] \leq 0$ according to the parameter $\gamma \in (0, 1)$. Combined with the iterative derivatives of the mixed term~\eqref{eqn: nag-4-mix}, the distance square~\eqref{eqn: nag-1-dist} and the potential function~\eqref{eqn: nag-2-pot}, it follows the iterative difference of the Lyapunov function~\eqref{eqn: lyapunov-nag} as
\begin{align}
\mathcal{E}(k+1) - \mathcal{E}(k)  \leq &  - \frac{s^{\gamma+1}E_{k}}{4}\|\nabla f(y_{k})\|^2 + \frac{s^{\gamma+1}G_{k}}{4} \|\nabla f(y_{k-1})\|^2 \nonumber \\
                                        &  + \frac{s^{\gamma}}{2}E_{k}\langle\nabla f(y_{k}), y_{k} - x^\star\rangle -\frac{s^{\gamma}(G_{k+1} - G_{k})}{2} \left(f(y_{k}) - f(x^\star)\right) \nonumber \\
                                     =   &  \frac{s^{\gamma}}{2}E_{k} \left[ \langle\nabla f(y_{k}), y_{k} - x^\star\rangle - \frac{s}{2}\|\nabla f(y_{k})\|^2 - \left( f(y_{k}) - f(x^\star)\right) \right] \nonumber \\
                                        & - \frac{s^{\gamma}(G_{k+1} - G_{k} - E_{k})}{2} \left(f(y_{k}) - f(x^\star)\right)  + \frac{s^{\gamma+1}G_{k}}{4} \|\nabla f(y_{k-1})\|^2.    \label{eqn: lyapunov-nag-proof-1}
\end{align}
Putting the iterative variable $y_{k}$ and the unique minimizer $x^\star$ into the gradient Lipschitz inequality~\eqref{eqn: grad-lip}, we have
\begin{equation}
\label{eqn: grad-lip2}
f(y_{k}) - f(x^\star) \leq  \langle\nabla f(y_{k}), y_{k} - x^\star\rangle - \frac{s}{2}\|\nabla f(y_{k})\|^2,
\end{equation}
which is due to the step size satisfies $0 < s \leq 1/L$. 
For the coefficients,
\begin{align}
G_{k+1} - G_{k} - E_{k} & = -2(k+1)^{\gamma}\left[(k+1)^{\gamma} + 2\gamma(k+2)^{\gamma-1} \right] +2 \left[k^{\gamma} + 2\gamma(k+1)^{\gamma-1} \right]^2 \nonumber \\
                        & = \gamma k^{2\gamma-2} + O\left( k^{2\gamma-3}\right).  \label{eqn: lyapunov-nag-proof-2}
\end{align} 
Let $K_0:=K_0(\gamma) = \max\{ K_1(\gamma), K_2(\gamma) \} $. With~\eqref{eqn: lyapunov-nag-proof-1},~\eqref{eqn: grad-lip2} and~\eqref{eqn: lyapunov-nag-proof-2}, we have 
\[
\mathcal{E}(k+1) - \mathcal{E}(k) \leq - \frac{s^{\gamma+1}k^{\gamma} \left( k^{\gamma} + 2\gamma(k+1)^{\gamma-1} \right)}{2} \|\nabla f(y_{k-1})\|^2,
\]
for any $k \geq K_0$.  Therefore, we obtain the convergence rate of the objective value~\eqref{eqn: nag-objvalue} and the following limitation as
\[
\lim_{k\rightarrow \infty} \sum_{i=k}^{\infty} s^{\gamma+1} i^{2\gamma} \| \nabla f(y_{i-1}) \|^2 = 0,
\]
which implies the convergence rate of the gradient norm square~\eqref{eqn: nag-gradnorm} as
\[
0 \leq \lim_{k\to\infty}\frac{s^{\gamma+1}k^{2\gamma+1}}{2^{2\gamma}}\min_{0\leq i\leq k}\|\nabla f(y_{i})\|^{2}\leq\lim_{k\to\infty}\sum_{i=\lceil\frac{k}{2}\rceil}^{k+1}s^{\gamma+1}i^{2\gamma}\|\nabla f(y_{i-1})\|^{2}=0.
\]
This completes the proof. 
\end{proof}

\section{Generalization to the composite optimization}
\label{sec: proximal}

In this section, we generalize the convergence rates obtained in~\Cref{thm: conv-rate-nag} to the composite optimization by use of the improved fundamental proximal gradient inequality~\eqref{eqn: prox-grad} instead of the gradient Lipschitz inequality~\eqref{eqn: grad-lip}.

Putting the iterative sequence $x_{k+1}$ and $x_{k}$ into the improved fundamental proximal gradient inequality~\eqref{eqn: prox-grad}, we have
\begin{equation}
\label{eqn: prox-grad-1}
\Phi(x_{k+1}) - \Phi(x_k) \leq \left\langle G_s(y_k), y_k - x_k \right\rangle - \left(s - \frac{Ls^2}{2} \right) \|G_s(y_k)\|^2.
\end{equation}
Consequently, we use $\Phi(x_{k}) - \Phi(x^\star)$ to take place of $f(y_{k-1}) - \Phi(x^\star)$ in the Lyapunov function~\eqref{eqn: lyapunov-nag} as
\begin{align}
\mathcal{E}(k)  = & s^{\gamma}k^{\gamma}\left[k^{\gamma} + 2\gamma(k+1)^{\gamma-1}\right] \left( \Phi(x_{k}) - \Phi(x^\star) \right)  \nonumber \\
                  & + \frac12 \big\|(k-1)^{\gamma}s^{\gamma/2}v_{k} + 2\gamma k^{\gamma-1}s^{(\gamma-1)/2}(x_{k} - x^\star)\big\|^2 \nonumber \\
                  & + 2 \gamma s^{\gamma-1} \left\{ \frac{(k-2)k^{\gamma-1}[(k-1)^{\gamma} + 2\gamma k^{\gamma-1}]}{k + 3\gamma -2} - (k-1)^{\gamma-1}(k-2)^{\gamma} \right\} \|x_{k-1} - x^\star\|^2. \label{eqn: lyapunov-fista}
\end{align}
Here, the unique difference from the calculations based on the gradient Lipschitz inequality shown in the last section is in the third step, the iterative difference of the potential function.  With the inequality above~\eqref{eqn: prox-grad-1}, we estimate the iterative difference as
\begin{align}
    - \frac{s^{\gamma}G_{k+1}}{2} (\Phi(x_{k+1}) - \Phi(x^\star)) + &\frac{s^{\gamma}G_{k}}{2} (\Phi(x_{k}) - \Phi(x^\star))\nonumber \\
\leq  & - \underbrace{\frac{s^{\gamma}G_{k}}{2} \left[ \left\langle G_s(y_{k}), y_{k} - x_{k} \right\rangle - \left(s - \frac{Ls^2}{2} \right) \|G_s(y_k)\|^2 \right]}_{\textbf{III}_1} \nonumber  \\
 &- \underbrace{\frac{s^{\gamma}(G_{k+1} - G_{k})}{2} \left(\Phi(x_{k+1}) - \Phi(x^\star)\right)}_{\textbf{III}_2}. \label{eqn: fista-2-pot}
\end{align}
Furthermore, putting the iterative sequence $x_{k+1}$ and $x^\star$ into the improved fundamental proximal gradient inequality~\eqref{eqn: prox-grad}, we have
\begin{equation}
\label{eqn: prox-grad-2}
\Phi(x_{k+1}) - \Phi(x^\star) \leq \left\langle G_s(y_k), y_k - x^\star \right\rangle - \frac{s}{2}\|G_s(y_k)\|^2,
\end{equation}
where the inequality is simplified due to the step size satisfying $0 < s \leq L/2$. With~\eqref{eqn: fista-2-pot} and~\eqref{eqn: prox-grad-2}, we obtain the iterative difference of the Lyapunov function~\eqref{eqn: lyapunov-fista} as
\[
\mathcal{E}(k+1) - \mathcal{E}(k) \leq - \frac{s^{\gamma+1}(1-sL)k^{\gamma} \left( k^{\gamma} + 2\gamma(k+1)^{\gamma-1} \right)}{2} \|G_s(y_{k})\|^2.
\]
for any $k \geq K_0(\gamma)$. Hence, we conclude this section by characterizing the convergence rates for the composite optimization with the following theorem rigorously. 
\begin{theorem}
\label{thm: conv-rate-fista}
Let $f \in \mathcal{F}_{L}^{1}(\mathbb{R}^d)$ and $g \in \mathcal{F}^0(\mathbb{R}^d)$. There exists $K_{0}:=K_0(\gamma)$ only depending on $\gamma \in(-1,2)$ such that the iterative sequence $\{y_{k}\}_{k=0}^{\infty}$ generated by FISTA obeys
\begin{equation}
\label{eqn: fista-gradnorm}
    \lim_{k\to\infty}\left[s^{\gamma+1}k^{2\gamma+1}\min_{0\leq i\leq k}\|G_s(y_{i})\|^{2}\right] = 0
\end{equation}
for any $0 < s < 1/L$. When the iteration number satisties $k \geq K_0$,  the objective value converges with the following rate as
\begin{equation}
\label{eqn: fista-objvalue}
\Phi(x_{k}) - \Phi(x^\star)\leq\frac{\mathcal{E}(K_{0})}{s^{\gamma}k^{2\gamma}}
\end{equation}
for any $0<s \leq 1/L$. 
\end{theorem}
\section{The critical case}
\label{sec: critical-case}

In this section, we discuss the convergence behavior of~\texttt{NAG} with its proximal scheme under the critical case $r=-1$, or $\gamma=0$, under the high-resolution differential equation framework, which has been unraveled via the low-resolution differential equation. 
\paragraph{High-resolution differential equation}

First, we consider the continuous high-resolution differential equation. When the parameter satisfies the critical case, that is, $\gamma =0$, the Lyapunov function~\eqref{eqn: lya-con} is reduced to
\begin{equation}
\label{eqn: con-critical-1}
\mathcal{E}(t) = f(X+\sqrt{s}\dot{X}) - f(x^\star) + \frac{1}{2}\|\dot{X}\|^{2},
\end{equation}
with the time derivative as
\begin{equation}
\label{eqn: con-critical-2}
\frac{d\mathcal{E}}{dt} = \langle\nabla f(X+\sqrt{s}\dot{X}), \dot{X} + \sqrt{s}\ddot{X}\rangle + \langle\dot{X}, \ddot{X}\rangle = - \sqrt{s}\|\nabla f(X+\sqrt{s}\Dot{X})\|^{2} \leq 0.
\end{equation}
With~\eqref{eqn: con-critical-1} and~\eqref{eqn: con-critical-2}, we formalize the convergence behavior with the following theorem. 
\begin{theorem}
\label{thm:con-critical}
Let $f\in \mathcal{F}_1^L(\mathbb{R}^d)$, then the solution $X=X(t)$ to the high-resolution differential equation~\eqref{eqn: high-ode} with the parameter $\gamma = 0$ obeys
\begin{equation}
\label{eqn: con-gradnorm-critical}
    \lim_{t\to\infty} \left( t\|\nabla f(X+\sqrt{s}\dot{X})\|^2\right)=0,
\end{equation}
and the objective value $f(X+\sqrt{s}\dot{X}) - f(x^\star)$ is bounded for all $t\geq 0$. 
\end{theorem}

\paragraph{Nesterov's acceleration}

With the new iterative sequence $v_{k} = (x_{k} - x_{k-1})/\sqrt{s}$, we write down the phase-space representation of~\texttt{NAG} with the parameter $\gamma = 0$ as
\begin{equation}
\label{eqn: phase-critical}
\left\{
\begin{aligned}
& x_{k+1} - x_{k} = \sqrt{s}v_{k+1}, \\
&v_{k+1} - v_{k} = -\sqrt{s}\nabla f(y_{k}),
\end{aligned}
\right.
\end{equation}
where the iterative sequence $y_k$ satisfies the following relation as
\begin{equation}
\label{eqn: yxv-critical}
y_{k} = x_{k} + \sqrt{s}v_{k}.
\end{equation}
When the parameter satisfies the critical case, that is, $\gamma =0$, the Lyapunov function~\eqref{eqn: lyapunov-nag} is reduced to
\begin{equation}
\label{eqn: lyapunov-nag-critical}
\mathcal{E}(k) = f(y_{k-1}) - f(x^\star) + \frac{1}{2}\|v_{k}\|^{2}.
\end{equation}
With the phase-space representation~\eqref{eqn: phase-critical} and the relation among the variables~\eqref{eqn: yxv-critical}, we calculate the iterative difference as
\begin{align}
\mathcal{E}(k+1) - \mathcal{E}(k) = & f(y_{k}) - f(y_{k-1}) + \langle v_{k}, v_{k+1} - v_{k}\rangle + \frac{1}{2}\|v_{k+1} - v_{k}\|^{2} \nonumber \\
= &  f(y_{k}) - f(y_{k-1}) -\sqrt{s}\langle v_{k}, \nabla f(y_{k})\rangle + \frac{s}{2}\|\nabla f(y_{k})\|^{2} \nonumber \\
= &  f(y_{k}) - f(y_{k-1}) + \langle x_{k} - y_{k}, \nabla f(y_{k})\rangle + \frac{s}{2}\|\nabla f(y_{k})\|^{2}. \label{eqn: iter-nag-1}
\end{align}
With the inequality~\eqref{eqn: grad-lip-1}, we further estimate the iterative difference~\eqref{eqn: iter-nag-1} as
\begin{align}
\mathcal{E}(k+1) - \mathcal{E}(k) \leq & \langle\nabla f(y_{k}), x_{k} - y_{k-1}\rangle - \frac{1}{2L}\|\nabla f(y_{k}) - \nabla f(y_{k-1})\|^{2}  + \frac{s}{2}\|\nabla f(y_{k})\|^{2} \nonumber \\
= & \langle\nabla f(y_{k}), \nabla f(y_{k-1})\rangle - \frac{1}{2L}\|\nabla f(y_{k}) - \nabla f(y_{k-1})\|^{2}  + \frac{s}{2}\|\nabla f(y_{k})\|^{2} \nonumber \\
\leq & -\frac{s}{2}\|\nabla f(y_{k-1})\|^2, \label{eqn: iter-nag-2}
\end{align}
where the second equality follows the gradient step of~\texttt{NAG} and the last inequality follows the step size $0 < s \leq 1/L$. Hence, we show the following theorem to characterize the convergence behavior of~\texttt{NAG} with the parameter $\gamma=0$.  
\begin{theorem}
\label{thm: conv-rate-nag-critical}
Let $f \in \mathcal{F}_{L}^{1}(\mathbb{R}^d)$. The iterative sequence $\{y_{k}\}_{k=0}^{\infty}$ generated by \texttt{NAG} with the parameter $\gamma=0$  obeys
\begin{equation}
\label{eqn: nag-gradnorm-critical}
    \lim_{k\to\infty}\left(sk\min_{0\leq i\leq k}\|\nabla f(y_{i})\|^{2}\right) = 0,
\end{equation}
for any $0 < s \leq 1/L$ and the objective value $f(y_k) - f(x^\star)$ is bouned for all $k\geq 0$. Especially, let the step size be set as $s=1/L$, then we have
\[
\lim_{k\to\infty}\left(\frac{k}{L}\min_{0\leq i\leq k}\|\nabla f(y_{i})\|^{2}\right) = 0.
\]
\end{theorem}

\paragraph{Generalization to the composite optimization}

Similarly, with the new sequence $v_{k} = (x_{k} - x_{k-1})/\sqrt{s}$, we write down the phase-space representation of FISTA with the parameter $\gamma = 0$ as  
\begin{equation}
\label{eqn: phase-critical-fista}
\left\{
\begin{aligned}
& x_{k+1} - x_{k} = \sqrt{s}v_{k+1}, \\
&v_{k+1} - v_{k} = -\sqrt{s}\nabla f(y_{k}),
\end{aligned}
\right.
\end{equation}
where the iterative sequence $y_k$ also satisfies the following relation~\eqref{eqn: yxv-critical}. 
When the parameter satisfies the critical case, that is, $\gamma =0$, the Lyapunov function~\eqref{eqn: lyapunov-fista} is reduced to
\begin{equation}
\label{eqn: lyapunov-fista-critical}
\mathcal{E}(k) = \Phi(x_{k}) - \Phi(x^\star) + \frac{1}{2}\|v_{k}\|^{2},
\end{equation}
with the iterative difference between $\mathcal{E}(k+1)$ and $\mathcal{E}(k)$ as
\begin{equation}
\label{eqn: iter-fista-1}
\mathcal{E}(k+1) - \mathcal{E}(k) =   \Phi(x_{k+1}) - \Phi(x_{k})+ \langle x_{k} - y_{k}, G_s(y_{k})\rangle + \frac{s}{2}\|G_s(y_{k})\|^{2},
\end{equation}
which follows the phase-space representation~\eqref{eqn: phase-critical-fista} and the relation among the variables~\eqref{eqn: yxv-critical}. With the inequality~\eqref{eqn: prox-grad-1}, we further estimate the iterative difference~\eqref{eqn: iter-fista-1} as
\begin{equation}
\label{eqn: iter-nag-2}
\mathcal{E}(k+1) - \mathcal{E}(k) \leq -  \frac{s(1-sL)}{2}\|G_{s}(y_{k})\|^{2}.
\end{equation}
Hence, we show the following theorem to characterize the convergence behavior of FISTA  with the parameter $\gamma=0$.  
\begin{theorem}
\label{thm: conv-rate-fista-critical}
Let $f \in \mathcal{F}_{L}^{1}(\mathbb{R}^d)$ and $g \in \mathcal{F}^0(\mathbb{R}^d)$. The iterative sequence $\{y_{k}\}_{k=0}^{\infty}$ generated by FISTA with $\gamma=0$ obeys
\begin{equation}
\label{eqn: fista-gradnorm-critical}
    \lim_{k\to\infty}\left[sk\min_{0\leq i\leq k}\|G_s(y_{i})\|^{2}\right] = 0
\end{equation}
for any $0 < s < 1/L$ and the objective value $\Phi(x_k) - \Phi(x^\star)$ is bouned for all $k\geq 0$.
\end{theorem}

\section{Conclusion and discussion}
\label{sec: conclusion}

In this paper, we expand the high-resolution differential equation framework for~\texttt{NAG} and its proximal correspondence --- FISTA to the underdamped case. Motivated by the power of the time $t^{\gamma}$ or the iteration $k^{\gamma}$, the new Lyapunov functions are constructed to obtain the convergence rates, which do not only reproduce the previous rate of the objective value based on the low-resolution differential equation framework but also characterize the convergence rate of the gradient norm square. All the convergence rates obtained for the underdamped case are dependent on the parameter $\gamma \in (0,1)$. When the parameter is reduced to $\gamma=1$, the convergence rates of both the objective value and the gradient norm square are identical to that previously obtained in~\citep{chen2022gradient}. In the numerical experiments, we also observe that~\texttt{NAG} converges for the critical case ($\gamma=0$), where the iterative behavior can be simulated by the high-resolution differential equation but is inconsistent with the low-resolution differential equation that degenerates to the conservative Newton's equation. Finally, we use the high-resolution differential equation framework to theoretically characterize the convergence rates, which are consistent with that obtained for the underdamped case with $\gamma=0$.  

In~\Cref{thm:con-critical}, we show that the Lyapunov function~\eqref{eqn: con-critical-1} along the solution to the high-resolution differential equation~\eqref{eqn: hig-ode-critical} decreases monotonously. Meanwhile, we find that the Lyapunov function for the critical case~\eqref{eqn: con-critical-1} is similar to the energy. Image from the view of physics, when a spring oscillates with friction, the energy is always dissipated to zero. Furthermore, it is shown in~\eqref{eqn: con-critical-2} that the time derivative of the Lyapunov function is dominated by the gradient norm square, which manifests that the Lyapunov function along the solution to the high-resolution differential equation must stop at the points where the gradient equals zero. The same phenomenon is also shown in~\Cref{thm: con} for the underdamped case and~\citep[Theorem 4.1]{chen2022gradient} for the critically damped case, but it does not appear along the solution to the low-resolution differential equation in~\citep[(18)]{su2016differential}. The description above manifests that the Lyapunov functions constructed probably converge to zero. However, it still needs to prove further at the stop points that the velocity equals zero or the position equals the unique minimizer $x^\star$.  If the Lyapunov functions constructed for any $\gamma \in [0,1]$ converge to zero, then the objective value should converge faster. Hence, we mention the following question as an open problem.
\begin{tcolorbox}
Does the~\texttt{NAG} or FISTA converge with the following rate
\[
f(y_{k}) - f(x^\star) \leq o \left(\left(\frac{1}{sk^2}\right)^{\frac{r+1}{3}}\right) \quad \text{and} \quad f(x_{k}) - f(x^\star) \leq o \left(\left(\frac{1}{sk^2}\right)^{\frac{r+1}{3}}\right)
\]
for any $r \in [-1,2]$?
\end{tcolorbox}

%
%


{\small
\subsection*{Acknowledgments}
We would like to thank Bowen Li for his helpful discussions.
\bibliographystyle{abbrvnat}
\bibliography{ref}
}
\end{document}